\documentclass[12pt]{amsart}

\usepackage{amssymb, amsfonts}
\usepackage{url}
\usepackage[all,arc]{xy}
\usepackage{amscd}
\usepackage{mathrsfs}
\usepackage{geometry}
\usepackage[colorlinks,citecolor=blue]{hyperref}
\usepackage{comment}
\usepackage{enumerate}
\usepackage{graphicx}
\usepackage{bm}
\usepackage{color}
\usepackage{anyfontsize}
\usepackage{caption}

\numberwithin{equation}{section}

\theoremstyle{plain}
\newtheorem{theorem}{Theorem}[section]
\newtheorem{corollary}[theorem]{Corollary}

\newtheorem{proposition}[theorem]{Proposition}

\theoremstyle{remark}
\newtheorem{remark}{Remark}[section]
\newtheorem{example}[remark]{Example}

\begin{document}


\title{Spectral and Nodal Properties of Static Manifolds with Boundary}

\author{Vladimir Medvedev}

\address{Faculty of Mathematics, National Research University Higher School of Economics, 6 Usacheva Street, Moscow, 119048, Russian Federation}

\email{vomedvedev@hse.ru}



\begin{abstract}
We study compact static manifolds with boundary from a spectral point of view, with particular emphasis on the relation between Dirichlet, Neumann, and Robin spectra and the nodal structure of static potentials. We establish comparison inequalities linking Dirichlet and Neumann eigenvalues with the scalar curvature and derive sufficient spectral conditions ensuring that a static potential changes sign. We also obtain bounds on the number of connected components of the zero set, sufficient conditions for its connectedness, and sufficient conditions ensuring that its components intersect the boundary. Finally, we study variational properties of Robin static triples and characterize compact Robin static triples whose metrics are critical for the first mixed Robin--Dirichlet eigenvalue under admissible volume-preserving deformations.
\end{abstract}

\maketitle


\newcommand\cont{\operatorname{cont}}
\newcommand\diff{\operatorname{diff}}

\newcommand{\dvol}{\text{dA}}
\newcommand{\Ric}{\operatorname{Ric}}
\newcommand{\Hess}{\operatorname{Hess}}
\newcommand{\GL}{\operatorname{GL}}
\newcommand{\myO}{\operatorname{O}}
\newcommand{\myP}{\operatorname{P}}
\newcommand{\eye}{\operatorname{Id}}
\newcommand{\myF}{\operatorname{F}}
\newcommand{\Vol}{\operatorname{Vol}}
\newcommand{\odd}{\operatorname{odd}}
\newcommand{\even}{\operatorname{even}}
\newcommand{\ol}{\overline}
\newcommand{\mye}{\operatorname{E}}
\newcommand{\myo}{\operatorname{o}}
\newcommand{\myt}{\operatorname{t}}
\newcommand{\irr}{\operatorname{Irr}}
\newcommand{\mydiv}{\operatorname{div}}
\newcommand{\curl}{\operatorname{curl}}
\newcommand{\re}{\operatorname{Re}}
\newcommand{\im}{\operatorname{Im}}
\newcommand{\can}{\operatorname{can}}
\newcommand{\scal}{\operatorname{scal}}
\newcommand{\tr}{\operatorname{trace}}
\newcommand{\sgn}{\operatorname{sgn}}
\newcommand{\SL}{\operatorname{SL}}
\newcommand{\myspan}{\operatorname{span}}
\newcommand{\mydet}{\operatorname{det}}
\newcommand{\SO}{\operatorname{SO}}
\newcommand{\SU}{\operatorname{SU}}
\newcommand{\specl}{\operatorname{spec_{\mathcal{L}}}}
\newcommand{\fix}{\operatorname{Fix}}
\newcommand{\id}{\operatorname{id}}
\newcommand{\grad}{\operatorname{grad}}
\newcommand{\singsup}{\operatorname{singsupp}}
\newcommand{\wave}{\operatorname{wave}}
\newcommand{\ind}{\operatorname{ind}}
\newcommand{\mynull}{\operatorname{null}}
\newcommand{\inj}{\operatorname{inj}}
\newcommand{\arcsinh}{\operatorname{arcsinh}}
\newcommand{\Spec}{\operatorname{Spec}}
\newcommand{\Ind}{\operatorname{Ind}}
\newcommand{\Nul}{\operatorname{Nul}}
\newcommand{\inrad}{\operatorname{inrad}}
\newcommand{\mult}{\operatorname{mult}}
\newcommand{\Length}{\operatorname{Length}}
\newcommand{\Area}{\operatorname{Area}}
\newcommand{\Ker}{\operatorname{Ker}}
\newcommand{\floor}[1]{\left \lfloor #1  \right \rfloor}

\newcommand\restr[2]{{
  \left.\kern-\nulldelimiterspace 
  #1 
  \vphantom{\big|} 
  \right|_{#2} 
  }}


\section{Introduction}

Let $(M^n,g)$, $n\geqslant3$, be a smooth compact Riemannian manifold with nonempty boundary. We use the conventions
\[
\Delta_g=\operatorname{div}_g\nabla^g,\qquad B_g(X,Y)=g(\nabla_X\nu,Y),\qquad H_g=\operatorname{tr}_{\gamma}B_g,
\]
where $\nu$ is the outward unit normal and $\gamma=g|_{T\partial M}$.

The triple $(M,g,V)$ is called a \emph{static manifold with boundary} if $V\in C^\infty(M)$ is a nontrivial solution of
\begin{equation}\label{eq:static-boundary-system}
\begin{cases}
\operatorname{Hess}_gV-(\Delta_gV)g-V\operatorname{Ric}_g=0& \text{in }M,\\
(\partial_\nu V)\gamma-VB_g=0 & \text{on }\partial M.
\end{cases}
\end{equation}
The function $V$ is referred to as a \emph{static potential}.

The boundary condition in~\eqref{eq:static-boundary-system} first appeared in~\cite{cruz2019prescribing}, while an equivalent formulation was given in~\cite{almaraz2020mass}. The terminology \emph{static manifolds with boundary} was introduced in~\cite{AlmarazDeLima2025}. Following~\cite{sheng2026obata}, we call a static manifold with boundary \emph{simple} if both the manifold and its boundary are connected. 

The study of static manifolds with boundary is relatively recent, and the existing literature on the subject is still sufficiently limited to be listed in full; see~\cite{cruz2019prescribing,ho2020deformation,almaraz2020mass,huang2022scalar,cruz2023critical,cruz2023static,sheng2024localized,medvedev2024static,sheng2025static,medvedev2025some,sheng2026obata,medvedev2026dimension,medvedev2026static}. Recently, a generalization of static manifolds with boundary has been introduced in~\cite{almaraz2026mass}.

We recall the basic geometric properties established in \cite[Proposition~1]{cruz2023static}. The scalar curvature $R_g$ is constant, and each boundary component is totally umbilical
with constant mean curvature. The static equations imply
\begin{equation}\label{eq:static-interior-identities}
-\Delta_gV=\dfrac{R_g}{n-1}V,\qquad\operatorname{Hess}_gV=V\left(\operatorname{Ric}_g-\dfrac{R_g}{n-1}g\right),
\end{equation}
and
\begin{equation}\label{eq:static-boundary-identities}
\partial_\nu V=\dfrac{H_g}{n-1}V,\qquad B_g=\dfrac{H_g}{n-1}\gamma.
\end{equation}
The mean curvature may take different constant values on different boundary components. Hence, every static potential \(V\) satisfies the \textit{Robin problem}
\begin{equation*}
\left\{
   \begin{array}{rcl}
\Delta_gV&=&-\dfrac{R_g}{n-1}V\quad\mbox{in}\quad M,\\
\dfrac{\partial V}{\partial \nu}&=&\dfrac{H_g}{n-1}V\quad\mbox{on}\quad\partial M.
\end{array}
   \right.
\end{equation*}

In this paper, we study static manifolds with boundary through their spectral data, including the Dirichlet, Neumann, and Robin spectra. We show how this spectral information can be used to derive sufficient conditions determining whether the zero set of a static potential is empty, connected, or intersects the boundary; see Section~\ref{sec1}. More precisely, in Proposition~\ref{prop:static-spectral-comparison} we establish a general comparison inequality relating the Neumann and Dirichlet eigenvalues to the scalar curvature of a static manifold with boundary. As a consequence, Corollary~\ref{cor:dirichlet-sign-change} gives a sufficient condition in terms of the first Dirichlet eigenvalue ensuring that a static potential changes sign. Further inequalities involving Neumann eigenvalues and the scalar curvature, together with their consequences for the nodal structure of static potentials, are obtained in Proposition~\ref{prop:neumann-nodal-bound} and Corollary~\ref{cor:neumann-sign-nodal-criteria}. In Proposition~\ref{prop:zero-set-components}, we derive a general bound on the number of connected components of the zero set of a static potential and obtain sufficient conditions for this zero set to be connected. Proposition~\ref{prop:no-interior-nodal-domains} provides sufficient conditions ensuring that every component of the zero set intersects the boundary of the static manifold. By combining these results with those from our previous work~\cite[Theorem~1.11]{medvedev2025some}, we obtain Corollaries~\ref{cor:spectral-nodal-area} and~\ref{cor:spectral-nodal-disk}. Finally, Proposition~\ref{prop:cheeger-boundary-layer-sign-change} gives a sufficient condition for a static potential to change sign in terms of explicit geometric quantities.

Section~\ref{sec2} is devoted to the variational properties of Robin static triples, introduced in~\cite{medvedev2025some}. In particular, Theorem~\ref{prop:intrinsic-robin-static-criticality} characterizes those compact Robin static triples whose metrics are critical for the first mixed Robin--Dirichlet eigenvalue under volume-preserving deformations.

\section{Notation and conventions}

Throughout the paper, we use the following notation and conventions:

\begin{itemize}
    \item The \emph{nodal set} of the static potential $V$ is denoted by
    \[
    \Sigma:=\{x\in M:V(x)=0\}.
    \]

    \item Let $d(V)$ be the number of \emph{nodal domains} of $V$, namely, the number of connected components of $M\setminus\Sigma$, with connectedness understood in the relative topology of $M$.

    \item The symbol $m(V)$ stands for the number of connected components of $\Sigma$. By convention, $m(V)=0$ when $\Sigma=\varnothing$.
    
     \item Consequently, $d(V)=1$ if and only if $V$ is nowhere vanishing, whereas $\Sigma\neq\varnothing$ implies $d(V)\geqslant 2$.
    
\item The Dirichlet spectrum of a compact manifold with boundary $(M,g)$ is written as
\[
0<\lambda_1^D(M,g)<\lambda_2^D(M,g)\leqslant \lambda_3^D(M,g)\leqslant\ldots\nearrow+\infty,
\]
while its Neumann spectrum is written as
\[
0=\lambda_1^N(M,g)<\lambda_2^N(M,g)\leqslant \lambda_3^N(M,g)\leqslant\ldots\nearrow+\infty.
\]
All eigenvalues are counted with multiplicity, and each eigenvalue has finite multiplicity.
    
    \item For a nodal domain $\Omega$, the boundary of $\overline{\Omega}$ consists of the nodal part $\overline{\Omega}\cap\Sigma$ together with the part inherited from $\partial M$. Whenever these two parts intersect, they meet orthogonally, see~\cite[Proposition~1]{cruz2023static}.

    \item If $\overline{\Omega}\subset\operatorname{Int}(M)$ then $|V|$ is a positive Dirichlet eigenfunction on $\Omega$ corresponding to the eigenvalue $\dfrac{R_g}{n-1}$.

\end{itemize}

\section{Results}

\subsection{Spectral bounds and the zero set}\label{sec1}

In our previous papers~\cite{medvedev2024static,medvedev2026static}, we studied static manifolds with boundary under two assumptions on the static potential: either \(V\) is nowhere vanishing or its zero set is connected. In this section, we show how spectral information can be used to characterize situations in which these assumptions are automatically satisfied. More precisely, we derive conditions ensuring that \(V\) is nowhere vanishing and, when \(V\) does vanish, conditions guaranteeing that its zero set is connected.

As observed above, a static potential on a static manifold with boundary satisfies a Robin boundary value problem. In the particular case when the boundary is minimal, this Robin condition reduces to the Neumann condition. Hassannezhad and Sher~\cite{HassannezhadSher} also established a useful relation between the Robin and Dirichlet spectra, which we now recall.

For a real-valued function $b$ on $\partial M$, let $\lambda_k^R(M,g;b)$ denote the $k$-th eigenvalue of
\[
-\Delta_gu=\lambda u\quad\text{in }M,\qquad \partial_\nu u=bu\quad\text{on }\partial M.
\]
Its quadratic form is
\begin{equation}\label{eq:robin-quadratic-form}
Q_b(u)=\int_M|\nabla^gu|_g^2\,d\mu_g-\int_{\partial M}bu^2\,d\sigma_g,\qquad u\in H^1(M).
\end{equation}

If $b=\alpha=const$, we write for brevity  
\[
\lambda_k(\alpha)=\lambda_k^R(M,g;\alpha).
\]
By~\cite[Proposition~2.6]{HassannezhadSher}, the function $\alpha\mapsto\lambda_k(\alpha)$ is continuous and strictly decreasing. Moreover,
\begin{equation}\label{eq:robin-eigenvalues-monotonicity}
\lim_{\alpha\to-\infty}\lambda_k(\alpha)=\lambda_k^D(M,g).
\end{equation}

Our first result relies on this relation between the Robin and Dirichlet spectra.

\begin{proposition}\label{prop:static-spectral-comparison}
Let $(M,g,V)$ be a compact connected static manifold with boundary, and suppose that
\[
\dfrac{R_g}{n-1}=\lambda_k^R\left(M,g;\dfrac{H_g}{n-1}\right),
\]
i.e. $V$ is a $k$-th eigenfunction with eigenvalue $\dfrac{R_g}{n-1}$.
Then
\begin{equation}\label{eq:static-dirichlet-comparison}
\lambda_k^D(M,g)>\dfrac{R_g}{n-1}.
\end{equation}
In particular, if $V$ is nowhere vanishing, then
\begin{equation}\label{eq:positive-potential-dirichlet-gap}
\lambda_1^D(M,g)>\dfrac{R_g}{n-1}.
\end{equation}
If additionally $H_g\leqslant0$, then
\begin{equation}\label{eq:static-neumann-dirichlet-comparison}
\lambda_k^N(M,g)\leqslant \dfrac{R_g}{n-1}<\lambda_k^D(M,g).
\end{equation}
\end{proposition}

\begin{proof}
By~\eqref{eq:robin-eigenvalues-monotonicity}, $\lambda_k(\alpha)<\lambda_k^D(M,g)$ for every real constant $\alpha$. The min--max principle applied to \eqref{eq:robin-quadratic-form} gives
\[
\lambda_k^R\left(M,g;\dfrac{H_g}{n-1}\right)\leqslant\lambda_k\left(\dfrac{\min_{\partial M}H_g}{n-1}\right)<\lambda_k^D(M,g),
\]
proving \eqref{eq:static-dirichlet-comparison}.

If $V$ is nowhere vanishing, it is the first Robin problem with eigenvalue $\dfrac{R_g}{n-1}$, i.e. $k=1$.

If $H_g\leqslant0$, then $Q_{H_g/(n-1)}\geqslant Q_0$, so
\[
\lambda_k^N(M,g)=\lambda_k^R\left(M,g;0\right)\leqslant\lambda_k^R\left(M,g;\dfrac{H_g}{n-1}\right)=\dfrac{R_g}{n-1}.
\]
\end{proof}

\begin{remark}\label{rem:positive-scalar-curvature}
Notice that the assumption \(H_g\leqslant 0\) implies \(R_g\geqslant 0\); see, for instance,~\cite{cruz2019prescribing,sheng2025static}. Hence, the comparison in~\eqref{eq:static-neumann-dirichlet-comparison} is meaningful regardless of the sign of $R_g$.
\end{remark}

The preceding result immediately yields the following corollary.

\begin{corollary}\label{cor:dirichlet-sign-change}
Let $(M,g,V)$ be a compact connected static manifold with boundary such that $R_g>0$. If $\lambda_1^D(M,g)\leqslant \dfrac{R_g}{n-1}$, then $V$ changes sign. In particular, under the normalization $R_g=n(n-1)$,
\[
\lambda_1^D(M,g)\leqslant n \quad\Longrightarrow\quad V\text{ changes sign}.
\]
\end{corollary}

We now obtain a complementary criterion involving the Neumann spectrum. We have already obtained a result in Proposition~\ref{prop:static-spectral-comparison}. Below we show that it can be improved. 

For the nodal domains $\Omega_1,\ldots,\Omega_{d(V)}$ of $V$, define the boundary correction term
\[
\mathcal C_g(V)=\max_{1\leqslant j\leqslant d(V)}\dfrac{\displaystyle\int_{\overline{\Omega_j}\cap\partial M}H_gV^2\,d\sigma_g}{\displaystyle\int_{\Omega_j}V^2\,d\mu_g}.
\]

\begin{proposition}\label{prop:neumann-nodal-bound}
Let $(M^n,g,V)$ be a compact connected static manifold with boundary, and let $\Omega_1,\ldots,\Omega_{d(V)}$ be the nodal domains of $V$. Then
\[
\lambda_{d(V)}^N(M,g)\leqslant\dfrac{R_g+\mathcal C_g(V)}{n-1}
\]
In particular, if $H_g\leqslant0$, then
\[
\lambda_{d(V)}^N(M,g)\leqslant \dfrac{R_g}{n-1}.
\]
\end{proposition}

\begin{proof}
For each $j$, let $u_j$ equal $V$ on $\Omega_j$ and zero elsewhere on $M$. Since $V$ vanishes on the nodal face of $\Omega_j$, we have $u_j\in H^1(M,d\mu_g)$. Testing the first equation of~\eqref{eq:static-interior-identities} against $u_j$ and using $\partial_\nu V=\dfrac{H_g}{n-1}V$ on $\partial M$, we obtain
\begin{equation}\label{eq:identity}
(n-1)\int_M|\nabla^gu_j|_g^2\,d\mu_g=R_g\int_{\Omega_j}V^2\,d\mu_g+\int_{\overline{\Omega_j}\cap\partial M}H_gV^2\,d\sigma_g.
\end{equation}
The functions $u_1,\ldots,u_{d(V)}$ are linearly independent, and their products and gradient inner products vanish almost everywhere for distinct indices. Thus, setting $E=\operatorname{span}\{u_1,\ldots,u_{d(V)}\}$ the Neumann min--max principle yields
\[
\lambda_{d(V)}^N(M,g)\leqslant\sup_{u\in E\setminus\{0\}}\dfrac{\displaystyle\int_M|\nabla^gu|_g^2\,d\mu_g}{\displaystyle\int_Mu^2\,d\mu_g}=\max_{1\leqslant j\leqslant d(V)}\dfrac{\displaystyle\int_M|\nabla^gu_j|_g^2\,d\mu_g}{\displaystyle\int_Mu_j^2\,d\mu_g}.
\]
Substituting~\eqref{eq:identity} proves the first inequality. When $H_g\leqslant0$ on $\partial M$, every boundary integral is nonpositive, giving the stated particular case.
\end{proof}

\begin{remark}
If $R_g\leqslant 0$, then $H_g\geqslant 0$, which implies that $\mathcal C_g(V)\geqslant 0$; see, for instance,~\cite{cruz2019prescribing,sheng2025static}. Hence, the bound in the previous proposition is meaningful regardless of the sign of $R_g$.
\end{remark}

\begin{corollary}\label{cor:neumann-sign-nodal-criteria}
Let $(M^n,g,V)$ be a compact connected static manifold with
boundary. Then:
\begin{enumerate}
\item if $\lambda_2^N(M,g)>\dfrac{R_g+\mathcal C_g(V)}{n-1}$, then $V$ is nowhere vanishing;

\item if $V$ changes sign, then $\lambda_2^N(M,g)\leqslant\dfrac{R_g+\mathcal C_g(V)}{n-1}$. If, in addition, $\lambda_3^N(M,g)>\dfrac{R_g+\mathcal C_g(V)}{n-1}$, then $V$ has exactly two nodal domains. If moreover $H^1(M;\mathbb R)=0$, its nodal set is connected.
\end{enumerate}
\end{corollary}

\begin{proof}
Proposition~\ref{prop:neumann-nodal-bound} gives $\lambda_{d(V)}^N(M,g)\leqslant\dfrac{R_g+\mathcal C_g(V)}{n-1}$. If $V$ changes sign, then $d(V)\geqslant2$, which proves the asserted bound for $\lambda_2^N(M,g)$. Since every zero of a nontrivial static potential entails a sign change, the first conclusion follows by contraposition.

The additional inequality for $\lambda_3^N(M,g)$ excludes $d(V)\geqslant3$, so $d(V)=2$. When $H^1(M;\mathbb R)=0$, connectedness of the nodal set follows from Proposition~\ref{prop:zero-set-components} below.
\end{proof}

\begin{remark}\label{rem:negative-mean-curvature}
If \(H_g\leqslant 0\) on \(\partial M\), then \(\mathcal C_g(V)\leqslant 0\), while Remark~\ref{rem:positive-scalar-curvature} implies that \(R_g\geqslant 0\). In particular:
\begin{enumerate}
\item if $\lambda_2^N(M,g)>\dfrac{R_g}{n-1}$, then $V$ is nowhere vanishing;
\item if $V$ changes sign and $\lambda_3^N(M,g)>\dfrac{R_g}{n-1}$, then $V$ has exactly two nodal domains. If additionally $H^1(M;\mathbb R)=0$, its nodal set is connected.
\end{enumerate}
\end{remark}

\begin{remark}
The existence of exactly two nodal domains does not imply that the nodal set is connected, even for a compact simple static manifold with boundary and positive scalar curvature. Indeed, consider
\[
M=S^1\times S^2_+,\qquad g=dt^2+g_{S^2},\qquad V(t,x)=\sin t,
\]
where $S^1=\mathbb R/(2\pi\mathbb Z)$ and $S^2_+$ is the closed unit round hemisphere. Since $R_g=2$ and
\[
\operatorname{Hess}_gV=-V\,dt^2=V\left(\operatorname{Ric}_g-\dfrac{R_g}{2}g\right),
\]
the interior static equation holds. Moreover, $\partial M$ is connected and totally geodesic, and $\partial_\nu V=0$, so the static boundary condition is also satisfied. The potential has exactly two nodal domains,
\[
(0,\pi)\times S^2_+,\qquad (\pi,2\pi)\times S^2_+,
\]
whereas its nodal set
\[
\Sigma=\bigl(\{0\}\times S^2_+\bigr)\sqcup\bigl(\{\pi\}\times S^2_+\bigr)
\]
has two connected components. 
\end{remark}

The preceding results give complementary tests for the sign of $V$. Under $H_g\leqslant0$, $R_g\geqslant 0$ (see Remark~\ref{rem:positive-scalar-curvature}) and the first conclusion in Remark~\ref{rem:negative-mean-curvature}, followed by Proposition~\ref{prop:static-spectral-comparison}, gives
\[
\lambda_2^N(M,g)>\dfrac{R_g}{n-1}\quad\Longrightarrow\quad V\text{ is nowhere vanishing}\quad\Longrightarrow\quad\lambda_1^D(M,g)>\dfrac{R_g}{n-1}.
\]
Conversely, Corollary~\ref{cor:dirichlet-sign-change} and Remark~\ref{rem:negative-mean-curvature} together imply, still under $H_g\leqslant0$ and $R_g>0$,
\[
\lambda_1^D(M,g)\leqslant\dfrac{R_g}{n-1}\quad\Longrightarrow\quad V\text{ changes sign}\quad\Longrightarrow\quad\lambda_2^N(M,g)\leqslant\dfrac{R_g}{n-1}.
\]

The following proposition provides sufficient conditions for the zero set of $V$ to be connected.

\begin{proposition}\label{prop:zero-set-components}
Let $(M,g,V)$ be a compact connected static manifold with boundary. Let $b_1(M)=\dim H_1(M;\mathbb R)$. Then
\begin{equation}\label{eq:zero-set-topological-bound}
m(V)\leqslant d(V)-1+b_1(M).
\end{equation}
If $H^1(M;\mathbb R)=0$, then $m(V)=d(V)-1$.

Consequently, if $H_g\leqslant0$, then
\[
m(V)\leqslant N_N\left(\dfrac{R_g}{n-1}\right)-1+b_1(M),
\]
where
\[
N_N(t)=\#\{j\geqslant1:\lambda_j^N(M,g)\leqslant t\}
\]
is the Neumann counting function. In particular, if
\[
H^1(M;\mathbb R)=0, \qquad H_g\leqslant0,\qquad \lambda_3^N(M,g)>\dfrac{R_g}{n-1},
\]
and $V$ changes sign, then $\Sigma=V^{-1}(0)$ is connected.
\end{proposition}

\begin{proof}
Following the construction of the \emph{nesting graph} in~\cite{sarnak2019topologies}, we consider the graph $G$ whose vertices correspond to the nodal domains and whose edges correspond to the components of $\Sigma$. Each component of $\Sigma$ is connected and two-sided, with its two sides contained in a positive and a negative nodal domain. It therefore determines an edge joining these domains.

The graph is connected and satisfies
\[
b_1(G)=m(V)-d(V)+1.
\]
Choose disjoint tubular neighborhoods of the components of $\Sigma$, compatible with $\partial M$, i.e. transverse intervals through boundary points remain in the boundary. Collapsing the complementary regions to vertices and the normal intervals to edges gives a continuous map $M\to G$. The induced map on first homology is surjective: every edge loop can be lifted by crossing the corresponding hypersurfaces and joining consecutive crossings inside the appropriate nodal domains. Hence $b_1(G)\leqslant b_1(M)$, proving \eqref{eq:zero-set-topological-bound}.

If $H^1(M;\mathbb R)=0$, then $G$ is a tree and $m(V)=d(V)-1$. The remaining conclusions follow from Proposition~\ref{prop:neumann-nodal-bound} (see also Remark~\ref{rem:negative-mean-curvature}).
\end{proof}

In the study of static manifolds with boundary, it is also natural to ask under what conditions the zero set of a static potential, when nonempty, intersects the boundary. We obtain the following result.

\begin{proposition}\label{prop:no-interior-nodal-domains}
Let $(M,g,V)$ be a compact connected static manifold with boundary. Assume that $\lambda_1^D(M,g) \geqslant \dfrac{R_g}{n-1}$, a condition that is automatically satisfied when $R_g\leqslant 0$. Then every nodal domain $\Omega$ satisfies $\overline\Omega\cap\partial M\neq\varnothing$. If, in addition, $M$ is diffeomorphic to the closed ball $B^n$, then every component of $\Sigma=V^{-1}(0)$ meets $\partial M$. 
\end{proposition}

\begin{proof}
Suppose that $\overline\Omega\subset \operatorname{Int}(M)$. Its boundary consists of smooth components of $\Sigma$, and $|V|$ is a positive Dirichlet eigenfunction on $\Omega$. Therefore $\lambda_1^D(\Omega,g)=\dfrac{R_g}{n-1}$. Strict domain monotonicity gives $\lambda_1^D(M,g)<\dfrac{R_g}{n-1}$, a contradiction.

Now suppose $M\cong B^n$ and a component $\Sigma_0$ of $\Sigma$ is disjoint from $\partial M$. The Jordan--Brouwer separation theorem provides an enclosed
region $U$ with $\partial U=\Sigma_0$ and $\overline U\subset \operatorname{Int}(M)$. A nodal domain containing any point of $U\setminus\Sigma$ must remain in $U$, since it cannot cross $\Sigma_0$. Its closure is therefore contained in $\operatorname{Int}(M)$, contradicting the first assertion.
\end{proof}

\begin{remark}
Corollary~4.5 in~\cite{medvedev2026static} states that if the space of static potentials on $(M,g)$ has dimension greater than one and one connected component of $\Sigma$ intersects the boundary, then every other connected component of $\Sigma$ also intersects the boundary. By contrast, Proposition~\ref{prop:no-interior-nodal-domains} yields the same conclusion without any assumption on the dimension of the space of static potentials, provided that $M$ is diffeomorphic to the closed ball $B^n$ and $\lambda_1^D(M,g)\geqslant \dfrac{R_g}{n-1}$.
\end{remark}

Proposition~\eqref{prop:no-interior-nodal-domains} naturally raises the following question.

\medskip

\textit{Open question. Is every compact static three-manifold with boundary diffeomorphic to \(\overline{B^3}\) isometric to a geodesic ball in a simply connected space form?}

\medskip

\begin{remark}\label{rem:static-ball-rigidity}
The assertion that every compact static manifold with boundary diffeomorphic to a closed ball is isometric to a geodesic ball in a simply connected space form is false in dimensions
$4\leqslant n\leqslant12$. Indeed, the constructions of B\"ohm \cite{Bohm1998}, Nienhaus--Wink \cite{NienhausWink2025}, Buttsworth--Hodgkinson \cite{ButtsworthHodgkinson2026}, and Wang \cite{Wang2026} provide nonround Einstein metrics on $\mathbb S^{n+1}$ of the form
\[
h=dt^2+a(t)^2g_{\mathbb S^2}+b(t)^2g_{\mathbb S^{n-2}},\qquad \operatorname{Ric}_h=\Lambda h,\qquad \Lambda>0,
\]
with smooth collapses of the two spherical factors at opposite endpoints of the interval. These metrics yield static manifolds with boundary by reversing the construction of the \textit{associated singular Einstein manifold} (see e.g.~\cite{ambrozio2017static,cruz2023static,medvedev2026dimension}).
\[
g=dt^2+a(t)^2d\theta^2+b(t)^2g_{\mathbb S^{n-2}},\qquad V=a(t)\cos\theta,\qquad 0\leqslant\theta\leqslant\pi.
\]
After including the smooth endpoint collapses, the resulting manifold is diffeomorphic to $\overline{B^n}$. The interior static equation follows from the fact that $g+V^2d\varphi^2$ is locally isometric to $h$ wherever $V\neq0$, and extends across the nodal set by smoothness. Moreover, the boundary is totally geodesic and $\partial_\nu V=0$, so the full static boundary condition holds. These examples satisfy $R_g=(n-1)\Lambda>0$, and their potentials change sign. Their metrics do not have constant sectional curvature, since otherwise the associated Einstein metrics would also have constant sectional curvature.

For $n\geqslant13$, the above constructions do not currently provide counterexamples, and, to the best of our knowledge, the general rigidity question remains unresolved. The numerical evidence discussed in \cite{Wang2026} suggests that the corresponding nonround doubly warped Einstein metrics on $\mathbb S^{n+1}$ cease to exist for $n\geqslant13$. This is not a nonexistence theorem and does not exclude other constructions of static metrics on balls.

On the other hand, applying the local warped-product analysis underlying the Kobayashi--Lafontaine classification \cite{kobayashi1982differential,lafontaine1983geometrie}, together with the static boundary condition and the ball topology, yields the rigidity assertion whenever $g$ is locally conformally flat, in every dimension $n\geqslant3$. No sign assumption on $R_g$ or $V$ is needed.
\end{remark}

Combining the preceding results with~\cite[Theorem~1.11]{medvedev2025some}, we obtain the following.

\begin{corollary}\label{cor:spectral-nodal-area}
Let $(M^3,g,V)$ be a compact connected orientable static manifold with nonempty boundary, $R_g>0$, and $H^1(M;\mathbb R)=0$. Suppose that $\lambda_3^N(M,g)>\dfrac{R_g+\mathcal C_g(V)}{2}$, and that at least one of the following conditions holds:
\[
\lambda_1^D(M,g)\leqslant\dfrac{R_g}{2},\qquad\text{or}\qquad H_g\geqslant0.
\]
Then $V$ changes sign, has exactly two nodal domains, and $\Sigma=V^{-1}(0)$ is connected. Moreover:
\begin{enumerate}
\item If $H_g=0$ and $\Sigma\cap\partial M\neq\varnothing$, then $\Sigma$ is a totally geodesic free boundary disk and
\[
\operatorname{Vol}_g(\Sigma)<\dfrac{12\pi}{R_g}.
\]

\item
If $H_g\geqslant0$ and $\Sigma\subset\operatorname{Int}(M)$, then $\Sigma$ is a totally geodesic two-sphere and
\[
\operatorname{Vol}_g(\Sigma)<\dfrac{24\pi}{R_g}.
\]
\end{enumerate}
If, in addition, $\partial M$ is connected and $\lambda_1^D(M,g)\geqslant R_g/2$, then $\Sigma\cap\partial M\neq\varnothing$.
\end{corollary}

\begin{proof}
Suppose first that $V$ does not change sign. Replacing $V$ by $-V$ if necessary, we have $V>0$ on $M$. If $\phi$ is a positive first Dirichlet eigenfunction, Green's identity and the Hopf boundary lemma give
\[
\left(\lambda_1^D(M,g)-\dfrac{R_g}{2}\right)\int_M V\phi\,d\mu_g=-\int_{\partial M}V\,\partial_\nu\phi\,d\sigma_g>0.
\]
Thus the assumption $\lambda_1^D(M,g)\leqslant\dfrac{R_g}{2}$ forces $V$ to change sign. The alternative assumption $H_g\geqslant0$ yields the same conclusion, since integration of the first equation of~\eqref{eq:static-boundary-identities} gives
\[
R_g\int_M V\,d\mu_g+\int_{\partial M}H_gV\,d\sigma_g=0,
\]
which is impossible when $V>0$.

By Proposition~\ref{prop:neumann-nodal-bound},
\[
\lambda_{d(V)}^N(M,g)\leqslant\dfrac{R_g+\mathcal C_g(V)}{2}.
\]
By assumption $\lambda_3^N(M,g)>\dfrac{R_g+\mathcal C_g(V)}{2}$, therefore $d(V)\leqslant2$. Since $V$ changes sign, $d(V)=2$. The assumption $H^1(M;\mathbb R)=0$ implies that every connected component of the properly embedded, two-sided hypersurface $\Sigma$ separates $M$. Consequently, the number of connected components of $\Sigma$ is $d(V)-1=1$. Applying \cite[Theorem~1.11]{medvedev2025some} after rescaling yields the stated topological conclusions and area bounds. In the first case, equality in the corresponding non-strict bound would make $(M,g)$ a round hemisphere of sectional curvature $R_g/6$. However, on that hemisphere, $\lambda_3^N(M,g)=\dfrac{R_g}{2}$ contrary to the assumed gap, since $\mathcal C_g(V)=0$. Thus the first area bound is also strict.

Finally, suppose that $\partial M$ is connected and $\Sigma\subset\operatorname{Int}(M)$. One of the two nodal domains, say $\Omega$, then satisfies $\overline{\Omega}\subset\operatorname{Int}(M)$. Since $V$ has a fixed sign on $\Omega$ and vanishes on its boundary, $\lambda_1^D(\Omega,g)=\dfrac{R_g}{2}$. Strict domain monotonicity gives
\[
\lambda_1^D(M,g)<\lambda_1^D(\Omega,g)=\dfrac{R_g}{2},
\]
proving the final assertion.
\end{proof}

The following example illustrates the preceding corollary.

\begin{example}[The Nariai static manifold with boundary]\label{ex:nariai-spectral-nodal-area}
Let $a>0$ and consider
\[
M=\left[-\dfrac{\pi a}{2},\dfrac{\pi a}{2}\right]\times\mathbb S^2,\qquad g=dt^2+a^2g_{\mathbb S^2},\qquad V(t,x)=\sin(t/a),
\]
where $g_{\mathbb S^2}$ is the unit round metric. This is a compact connected orientable static manifold with boundary, and
\[
R_g=\dfrac{2}{a^2},\qquad H_g=0, \qquad \partial_\nu V=0,\qquad H^1(M;\mathbb R)=0.
\]
In particular, $\mathcal C_g(V)=0$.

By separation of variables, the Neumann eigenvalues are
\[
\dfrac{k^2+\ell(\ell+1)}{a^2},\qquad k,\ell\in\{0,1,2,\ldots\},
\]
with multiplicities inherited from the spherical harmonics.
Consequently,
\[
\lambda_1^N(M,g)=0,\qquad \lambda_2^N(M,g)=\dfrac{1}{a^2}=\dfrac{R_g}{2},\qquad\lambda_3^N(M,g)=\dfrac{2}{a^2}=R_g.
\]
The first Dirichlet eigenfunction is $\cos(t/a)$, so
\[
\lambda_1^D(M,g)=\dfrac{1}{a^2}=\dfrac{R_g}{2}.
\]
Thus both alternative assumptions in Corollary~\ref{cor:spectral-nodal-area} hold, and the spectral gap condition reads
\[
\lambda_3^N(M,g)=R_g>\dfrac{R_g+\mathcal C_g(V)}{2}=\dfrac{R_g}{2}.
\]

The potential has exactly two nodal domains, and its zero set is the connected totally geodesic sphere
\[
\Sigma=\{0\}\times\mathbb S^2\subset\operatorname{Int}(M).
\]
Its area is
\[
\operatorname{Vol}_g(\Sigma)=4\pi a^2=\dfrac{8\pi}{R_g}<\dfrac{24\pi}{R_g},
\]
which verifies the conclusion in part~(2).

Finally, although $\lambda_1^D(M,g)=R_g/2$, the nodal set does not meet the boundary. Since $\partial M$ has two connected components, this example shows that the connectedness assumption on $\partial M$ in the final assertion cannot be omitted.
\end{example}

As an immediate consequence of Corollary~\ref{cor:spectral-nodal-area}, we obtain the following.

\begin{corollary}\label{cor:spectral-nodal-disk}
Let $(M^3,g,V)$ be a compact connected orientable static manifold with connected nonempty boundary, $R_g>0$, $H_g=0$, and $H^1(M;\mathbb R)=0$.
If
\[
\lambda_1^D(M,g)=\dfrac{R_g}{2}<\lambda_3^N(M,g),
\]
then $V$ changes sign and its zero set is a connected totally geodesic free boundary disk satisfying
\[
\operatorname{Vol}_g(\Sigma)<\dfrac{12\pi}{R_g}.
\]
\end{corollary}

\begin{remark}\label{rem:closed-spectral-nodal-area}
An analogous statement holds for \textit{closed static manifolds}. Namely, let $(M^3,g,V)$ be a closed connected orientable static manifold
with $R_g>0$ and $H^1(M;\mathbb R)=0$.
Write the eigenvalues of $-\Delta_g$, counted with multiplicity, as
\[
0=\lambda_1(M,g)<\lambda_2(M,g)\leqslant\lambda_3(M,g)\leqslant\cdots\nearrow+\infty.
\]
If $\lambda_3(M,g)>\dfrac{R_g}{2}$ then $V$ has exactly two nodal domains, and $\Sigma=V^{-1}(0)$ is a connected totally geodesic two-sphere with
\[
\operatorname{Vol}_g(\Sigma)<\dfrac{24\pi}{R_g}.
\]

\begin{proof}
The assumptions imply that $V$ is an eigenfunction with eigenvalue $\lambda_2(M,g)=\dfrac{R_g}{2}$. Courant's nodal domain theorem imply that $V$ has exactly two nodal domains. As above, $H^1(M;\mathbb R)=0$ implies that $\Sigma$ is connected. Shen's theorem~\cite{shen1997note}, gives the spherical topology and the non-strict area bound. Equality would force $(M,g)$ to be a round sphere, for which $\lambda_3(M,g)=R_g/2$, contradicting the spectral gap.
\end{proof}

\end{remark}

The previous results show that, once the Dirichlet and Neumann eigenvalues of a static manifold with boundary are known, one can draw conclusions about the zero set of its static potentials. A natural next step is therefore to obtain estimates for the Dirichlet and Neumann eigenvalues of static manifolds with boundary in terms of some geometric quantities. We begin this program with the following result, which shows that, for \(R_g>0\), a boundary layer of fixed thickness and sufficiently small volume relative to the interior forces \(V\) to change sign. 

\begin{proposition}\label{prop:cheeger-boundary-layer-sign-change}
Let $(M,g,V)$ be a compact connected static manifold with boundary such that $R_g>0$. For $\rho>0$, set
\[
M_\rho=\{x\in M:\operatorname{dist}_g(x,\partial M)\geqslant\rho\}.
\]
Then
\begin{equation}\label{eq:cheeger-boundary-layer-upper}
\lambda_1^D(M,g)\leqslant\dfrac{\operatorname{Vol}_g(M\setminus M_\rho)}{\rho^2\operatorname{Vol}_g(M_\rho)}.
\end{equation}
Consequently,
\begin{equation}\label{eq:cheeger-boundary-layer-sign-change}
\dfrac{\operatorname{Vol}_g(M\setminus M_\rho)}{\rho^2\operatorname{Vol}_g(M_\rho)}\leqslant\dfrac{R_g}{n-1}\quad\Longrightarrow\quad V\text{ changes sign}.
\end{equation}
\end{proposition}

\begin{proof}
The Lipschitz function
\[
u_\rho(x)=\min\left\{\dfrac{\operatorname{dist}_g(x,\partial M)}{\rho},1\right\}
\]
belongs to $H_0^1(M,d\mu_g)$ and equals $1$ on $M_\rho$. Moreover,
\[
\int_M|\nabla^gu_\rho|_g^2\,d\mu_g\leqslant\dfrac{\operatorname{Vol}_g(M\setminus M_\rho)}{\rho^2},\qquad\int_Mu_\rho^2\,d\mu_g\geqslant\operatorname{Vol}_g(M_\rho).
\]
Its Rayleigh quotient proves \eqref{eq:cheeger-boundary-layer-upper}. The sign-change assertion follows from Proposition~\ref{cor:dirichlet-sign-change}.
\end{proof}

The construction of the test function is inspired by~\cite[Section~4.1]{harrell2021complementary}.

\subsection{Spectral criticality of Robin static triples}\label{sec2}

Following~\cite{medvedev2025some}, a \emph{Robin static triple} $(M,g,V)$ has a positive interior potential satisfying the static equation, with Dirichlet and Robin boundary parts $D$ and $S$:
\begin{equation}\label{eq:robin-static-system}
\begin{cases}
\operatorname{Hess}_gV-(\Delta_gV)g-V\operatorname{Ric}_g=0 & \text{in }\operatorname{Int}(M),\\
V=0 & \text{on }D,\\
(\partial_\nu V)\gamma-VB_g=0 & \text{on }S.
\end{cases}
\end{equation}
See \cite[Definition~1.7]{medvedev2025some}. Here we also allow the extension of this definition in which $D$ and $S$ are compact smooth boundary faces meeting orthogonally along their common edge. We assume that $g$ and $V$ are smooth up to the faces and their intersection, and that $V>0$ on $M\setminus D$. The case $S=\varnothing$ is a static triple (see~\cite{ambrozio2017static}).

The local arguments underlying the preliminary properties apply to these triples as well. In particular, $R_g$ is constant, each component of $D$ is totally geodesic, and $|\nabla^gV|_g$ is a positive constant on each such component. Indeed, $V=dV=0$ at a point would force $V$ to vanish along geodesics entering the interior, while $\operatorname{Hess}_gV=0$ on $D$ gives constancy of $|\nabla^gV|_g$ there (see e.g. the arguments in the proof of~\cite[Proposition~1]{cruz2023static}). Also relying on the proof of~\cite[Proposition~1]{cruz2023static}, we show that the mean curvature is constant on each Robin face. To see this directly, the boundary equations give, for $Y$ tangent to $S$,
\[
V\operatorname{Ric}_g(Y,\nu)=\operatorname{Hess}_gV(Y,\nu)=\dfrac{V}{n-1}Y(H_g).
\]
The contracted Codazzi equation gives
\[
\operatorname{Ric}_g(Y,\nu)=(\operatorname{div}_S B_g)(Y)-Y(H_g)=-\dfrac{n-2}{n-1}Y(H_g).
\]
Thus $Y(H_g)=0$ wherever $V>0$, and the conclusion extends to the edge by continuity.

For a fixed decomposition $\partial M=D\cup S$, define
\begin{equation}\label{eq:mixed-rayleigh-quotient}
\lambda_1^{D,R}\left(M,\widetilde g;\dfrac{H_g}{n-1}\right)=\inf_{\substack{w\in H^1(M)\setminus\{0\}\\\operatorname{Tr}_D w=0}}\dfrac{\displaystyle\int_M|\nabla^{\widetilde g}w|_{\widetilde g}^2\,d\mu_{\widetilde g}-\dfrac1{n-1}\int_S H_gw^2\,d\sigma_{\widetilde g}}{\displaystyle\int_Mw^2\,d\mu_{\widetilde g}}.
\end{equation}
Here $H_g$ is the mean curvature for the initial metric $g$ and is held fixed when $\widetilde g$ varies; it is not replaced by $H_{\widetilde g}$.

Notice that 
\begin{equation}\label{eq:robin-static-first-eigenvalue}
\lambda_1^{D,R}\left(M,g;\dfrac{H_g}{n-1}\right)=\dfrac{R_g}{n-1}.
\end{equation}
Indeed, set
\[
I=\int_MV^2\,d\mu_g,\qquad u=\dfrac{V}{\sqrt I}.
\]
Then $u$ is a normalized positive eigenfunction of the mixed problem with eigenvalue $R_g/(n-1)$. Orthogonality to a positive first eigenfunction proves \eqref{eq:robin-static-first-eigenvalue}. The first eigenvalue is simple.

\begin{theorem}\label{prop:intrinsic-robin-static-criticality}
Let $(M^n,g,V)$ be a compact connected Robin static triple with $D\neq\varnothing$, and write
\[
D=\Sigma_1\sqcup\cdots\sqcup\Sigma_\ell,\qquad\kappa_j=|\nabla^gV|_g\big|_{\Sigma_j}.
\]

Call a smooth metric variation $g_t$ admissible if
\[
g_0=g,\qquad\operatorname{Vol}_{g_t}(M)=\operatorname{Vol}_g(M),\qquad\dot g_0=\mathcal L_Xg,
\]
for some smooth vector field $X$ tangent to $S$. No tangency condition is imposed along $D$.

The metric $g$ is critical for $\lambda_1^{D,R}(M,\cdot;H_g/(n-1))$ under all admissible variations if and only if $\kappa_1=\cdots=\kappa_\ell$.
In particular, criticality holds when $D$ is connected.
\end{theorem}

This theorem implies that the associated Einstein manifold to a compact connected Robin static triple $(M^n,g,V)$ with $D\neq\varnothing$ has no singularities if and only if  $g$ is critical for $\lambda_1^{D,R}(M,\cdot;H_g/(n-1))$ under all admissible variations.

\begin{proof}
For an arbitrary smooth metric variation, put
\[
h=\dot g_0,\qquad\lambda(t)=\lambda_1^{D,R}\left(M,g_t;\dfrac{H_g}{n-1}\right).
\]
One has
\begin{equation}\label{eq:mixed-metric-first-variation}
\begin{aligned}
\lambda'(0)={}&-\int_M\langle du\otimes du,h\rangle_g\,d\mu_g+\dfrac12\int_M\left(|\nabla^gu|_g^2-\dfrac{R_g}{n-1}u^2\right)\operatorname{tr}_gh\,d\mu_g\\&-\dfrac1{2(n-1)}\int_S H_gu^2\operatorname{tr}_S h\,d\sigma_g,
\end{aligned}
\end{equation}
where as above $u=\dfrac{V}{\sqrt I}, I=\displaystyle\int_MV^2\,d\mu_g$. Indeed, let $g_t$ be a smooth family of metrics with $g_0=g$ and $h=\left.\dfrac{d}{dt}\right|_{t=0}g_t$. Consider a differentiable eigenbranch satisfying
\[
\begin{cases}
-\Delta_{g_t}u_t=\lambda(t)u_t & \text{in }M,\\
u_t=0 & \text{on }D,\\
\partial_{\nu_t}u_t=\alpha_tu_t & \text{on }S.
\end{cases}
\]
For the computation, we temporarily allow the Robin coefficient $\alpha_t$ to vary. Write
\[
u=u_0,\qquad\dot u=\left.\frac{d}{dt}\right|_{t=0}u_t,\qquad\dot\alpha_0=\left.\frac{d}{dt}\right|_{t=0}\alpha_t.
\]
Integration by parts gives
\begin{equation}\label{eq:mixed-energy-identity}
\int_M |du_t|_{g_t}^2\,d\mu_{g_t}=\lambda(t)\int_Mu_t^2\,d\mu_{g_t}+\int_S\alpha_tu_t^2\,d\sigma_{g_t}.
\end{equation}
Recall the variation formulas, with $u$ held fixed,
\[
\left.\frac{d}{dt}\right|_{t=0}|du|_{g_t}^2=-\langle du\otimes du,h\rangle_g,\qquad\left.\frac{d}{dt}\right|_{t=0}d\mu_{g_t}=\frac12\operatorname{tr}_gh\,d\mu_g,
\]
and
\[
\left.\frac{d}{dt}\right|_{t=0}d\sigma_{g_t}=\frac12\operatorname{tr}_Sh\,d\sigma_g.
\]
Since the Dirichlet portion is fixed, $\dot u=0$ on $D$. Testing the eigenvalue equation at $t=0$ with $\dot u$ yields
\[
\int_M\langle\nabla^gu,\nabla^g\dot u\rangle_g\,d\mu_g=\lambda(0)\int_Mu\dot u\,d\mu_g+\int_S\alpha_0u\dot u\,d\sigma_g.
\]
Consequently, all terms involving $\dot u$ cancel when differentiating \eqref{eq:mixed-energy-identity}, and we obtain
\begin{equation}\label{eq:mixed-two-parameter-first-variation}
\begin{aligned}
\lambda'(0)\int_Mu^2\,d\mu_g&+\int_S\dot\alpha_0u^2\,d\sigma_g={}-\int_M\langle du\otimes du,h\rangle_g\,d\mu_g\\&+\frac12\int_M\bigl(|\nabla^gu|_g^2-\lambda(0)u^2\bigr)\operatorname{tr}_gh\,d\mu_g-\frac12\int_S\alpha_0u^2\operatorname{tr}_Sh\,d\sigma_g.
\end{aligned}
\end{equation}
Taking
\[
\alpha_t\equiv\frac{H_g}{n-1},\qquad\lambda(0)=\frac{R_g}{n-1},\qquad\int_Mu^2\,d\mu_g=1,
\]
gives \eqref{eq:mixed-metric-first-variation}. Here $H_g/(n-1)$ is held fixed at its value for the reference metric $g$.

Further, introduce
\[
T=\left(|\nabla^gu|_g^2-\dfrac{R_g}{n-1}u^2\right)g-2\,du\otimes du.
\]
Since $\Delta_gu=-R_gu/(n-1)$ and $R_g$ is constant, $\operatorname{div}_gT=0$. For $h=\mathcal L_Xg$, integration by parts in \eqref{eq:mixed-metric-first-variation} yields 
\begin{equation}\label{eq:mixed-stress-boundary-formula}
\lambda'(0)=\int_{\partial M}T(X,\nu)\,d\sigma_g-\dfrac1{2(n-1)}\int_S H_gu^2\operatorname{tr}_S(\mathcal L_Xg)\,d\sigma_g.
\end{equation}

Write $f=g(X,\nu)$ on $D$. Since $u=0$ there, 
\[
T(X,\nu)=-f(\partial_\nu u)^2.
\]
On $S$, tangency of $X$ and the Robin condition give 
\[
T(X,\nu)=-\dfrac{2H_g}{n-1}uX(u),\qquad\dfrac12\operatorname{tr}_S(\mathcal L_Xg)=\operatorname{div}_S X.
\]
Because $H_g$ is constant on each Robin face, its contribution to \eqref{eq:mixed-stress-boundary-formula} is
\[
-\dfrac1{n-1}\int_S\operatorname{div}_S(H_gu^2X)\,d\sigma_g=0.
\]
There is no edge contribution, since $u=0$ on $D\cap S$. Consequently,
\begin{equation}\label{eq:intrinsic-static-first-variation}
\lambda'(0)=-\int_D f(\partial_\nu u)^2\,d\sigma_g=-\dfrac1I\sum_{j=1}^{\ell}\kappa_j^2\int_{\Sigma_j}f\,d\sigma_g.
\end{equation}
This identity holds for every metric variation with $\dot g_0=\mathcal L_Xg$ and $X$ tangent to $S$, whether or not the variation preserves volume.

The corresponding volume derivative is
\begin{equation}\label{eq:intrinsic-volume-first-variation}
\left.\dfrac{d}{dt}\right|_{t=0}\operatorname{Vol}_{g_t}(M)=\dfrac12\int_M\operatorname{tr}_g(\mathcal L_Xg)\,d\mu_g=\int_D f\,d\sigma_g.
\end{equation}
If $\kappa_1=\cdots=\kappa_\ell$, equations \eqref{eq:intrinsic-static-first-variation} and \eqref{eq:intrinsic-volume-first-variation} prove stationarity for all admissible variations.

Conversely, fix $i\neq j$ and choose a smooth vector field $X$, supported away from $S$ and the edges, such that
\[
\int_{\Sigma_i}f\,d\sigma_g=1,\qquad\int_{\Sigma_j}f\,d\sigma_g=-1,\qquad f=0\quad\text{on the remaining components}.
\]
Let $h=\mathcal L_Xg$ and $\widehat g_t=g+th$ for small $|t|$. Its first volume variation vanishes. Hence the volume-normalized family
\[
g_t=\left(\dfrac{\operatorname{Vol}_g(M)}{\operatorname{Vol}_{\widehat g_t}(M)}\right)^{2/n}\widehat g_t
\]
is admissible and still satisfies $\dot g_0=h$. Stationarity and \eqref{eq:intrinsic-static-first-variation} give $\kappa_i^2=\kappa_j^2$. Since each $\kappa_j$ is positive, this proves that $\kappa_1=\cdots=\kappa_\ell$.
\end{proof}

\begin{remark}
The admissible variations need not preserve staticity. The proposition concerns the specified Lie-derivative directions, rather than arbitrary volume-preserving metric variations. If $S=\varnothing$, the functional is $\lambda_1^D$; if $H_g=0$ on $S$, it is the first mixed Dirichlet--Neumann eigenvalue.
\end{remark}

As an immediate corollary of Theorem~\ref{prop:intrinsic-robin-static-criticality}, we obtain the following.

\begin{corollary}\label{prop:static-nodal-domain-criticality}
Let $(M,g,V)$ be a compact connected static manifold with boundary, and let $\Omega$ be a nodal domain with nonempty nodal face $D=\overline\Omega\cap\Sigma$. Set
\[
S=\overline{\Omega\cap\partial M},\qquad D=\Sigma_1\sqcup\cdots\sqcup\Sigma_\ell, \qquad \kappa_j=|\nabla^gV|_g\big|_{\Sigma_j}.
\]
Then
\[
\lambda_1^{D,R}\left(\overline\Omega,g;\dfrac{H_g}{n-1}\right)=\dfrac{R_g}{n-1}.
\]

Consider smooth deformations
$\Omega_t=\Phi_t(\Omega)$ induced by diffeomorphisms of $M$ preserving $\partial M$, with $\Phi_0=\operatorname{id}$. Impose Dirichlet conditions on $\Phi_t(D)$ and Robin
conditions
\[
\partial_\nu u=\dfrac{H_g}{n-1}u\quad\text{on }\Phi_t(S),
\]
keeping the ambient metric $g$ fixed.

The domain $\Omega$ is critical for the first mixed eigenvalue under all such volume-preserving deformations if and only if
\[
\kappa_1=\cdots=\kappa_\ell.
\]
In particular, a connected nodal face implies criticality. This includes the pure Dirichlet case $S=\varnothing$ and the mixed Dirichlet--Neumann case $H_g=0$ on $S$.
\end{corollary}

\subsection*{Acknowledgments} During the preparation of this manuscript, OpenAI's ChatGPT was used to assist in identifying existing examples of non-standard Einstein metrics on the sphere relevant to Remark~\ref{rem:static-ball-rigidity}, locating a reference for the test-function construction used in the proof of Proposition~\ref{prop:cheeger-boundary-layer-sign-change}, improving the exposition and clarity of the text, proofreading, and providing feedback on the presentation of mathematical arguments. The research ideas and all proofs were developed independently of AI tools. The author assumes full responsibility for the mathematical content and conclusions of the article. This article is an output of research project HSE-BR-2025-84, implemented as part of the Basic Research Program at HSE University.





\bibliography{mybib}
\bibliographystyle{alpha}

\end{document}